\documentclass{amsart}

\usepackage{color} 
\usepackage{verbatim}
\usepackage{amsfonts}
\usepackage{amsmath}
\usepackage{amssymb}
\usepackage{amsthm}
\usepackage{enumerate}
\usepackage[utf8]{inputenc}
\usepackage{mathabx,epsfig}
\usepackage[bookmarks=false]{hyperref}
\usepackage{paralist, tabularx}
\usepackage[cmtip,arrow]{xy}
\usepackage{pb-diagram,pb-xy}

\DeclareMathOperator{\Th}{Th}

\DeclareMathOperator{\Stab}{Stab}

\DeclareMathOperator{\cl}{cl}
\DeclareMathOperator{\tp}{tp}
\DeclareMathOperator{\Aff}{Aff}
\DeclareMathOperator{\SL}{SL}
\DeclareMathOperator{\GL}{GL}

\DeclareMathOperator{\charf}{char}
\newcommand{\Gm}{{\mathbb{G}_m}}
\newcommand{\GmD}{{\mathbb{G}_m^{0}}}
\newcommand{\GmT}{{\mathbb{G}_m^{00}}}
\newcommand{\Ga}{{\mathbb{G}_a}}

\newcommand{\GaT}{{\mathbb{G}_a^{00}}}
\newcommand{\CC}{{\mathfrak{C}}}
\newcommand{\CCC}{{\bar{\mathfrak{C}}}}
\newcommand{\Qp}{{\mathbb{Q}_p}}
\newcommand{\OK}{\mathcal{O}_K}
\newcommand{\OCC}{\mathcal{O}_\CC}
\newcommand{\LL}{\mathcal{L}}
\newcommand{\PP}{\mathbb{P}}

\newcommand{\Sdef}[1]{\bar S(\CC/{#1})}

\newcommand{\Sinv}[1]{S(\CC/{#1})}
\newcommand{\SinvX}[2]{S_{#1}(\CC/{#2})}
\newcommand{\Sfin}[1]{S^{#1}(\CC)}

\newcommand{\Sext}[1]{S_{ext}(#1)}
\newcommand{\SextX}[2]{S_{ext,#1}(#2)}

\newcommand{\QM}[4]{  \begin{pmatrix} {#1} & {#2} \\ {#3} & {#4} \end{pmatrix}}
\newcommand{\QMs}[4]{ \left( \begin{smallmatrix} {#1} & {#2} \\ {#3} & {#4} \end{smallmatrix} \right)}

\newenvironment{innerproof}
{\proof}
{\endproof}

\newtheorem{theorem}{Theorem}
\numberwithin{theorem}{section}
\newtheorem{lemma}[theorem]{Lemma}

\newtheorem{fact}[theorem]{Fact}
\newtheorem{proposition}[theorem]{Proposition}

\newtheorem{corollary}[theorem]{Corollary}
\newtheorem{definition}[theorem]{Definition}

\newtheorem*{claim2}{Claim}

\theoremstyle{definition}
\newtheorem{example}[theorem]{Example}
\newtheorem{remark}[theorem]{Remark}

\title{Topological dynamics and NIP fields}

\author{Grzegorz Jagiella}

\thanks{The author is supported by the National Science Centre, Poland grant no. 2018/31/B/ST1/00357}

\address{Instytut Matematyczny Uniwersytetu Wrocławskiego, pl. Grunwaldzki 2/4, 50-384 Wrocław, Poland}
\address{ORCID: \href{http://orcid.org/0000-0002-5504-5260}{0000-0002-5504-5260}}

\email{grzegorz.jagiella@math.uni.wroc.pl}



\begin{document}
	\begin{abstract}
		We study definable topological dynamics of some algebraic group actions over an arbitrary NIP field $K$. We show that the Ellis group of the universal definable flow of $\SL_2(K)$ is non-trivial if the multiplicative group of $K$ is not type-definably connected, providing a way to find multiple counterexamples to the Ellis group conjecture, particularly in the case of dp-minimal fields. We also study some structure theory of algebraic groups over $K$ with definable f-generics.
	\end{abstract}
	\maketitle

\section*{Introduction}

In this paper, we consider topological dynamics of some groups linear over a field with NIP. We explain the motivations for this research assuming that the reader is familiar with the basics of \emph{definable topological dynamics}. We refer to \cite{New1}, \cite{New2}, \cite{Pil} and \cite{Krup} as the main sources, although we will recall all the necessary notions in the preliminaries later. A group $G$ definable in a first-order structure $M$ can be studied through its \emph{universal definable flow over} $M$, namely the set of external types $\SextX{G}{M}$ on which $G(M)$ acts in a natural way. This action is a $G(M)$-flow in the classical sense. Definable topological dynamics aims to interpret various classical dynamical notions (such as \emph{almost periodic type}s or the \emph{Ellis group}) in this setup. It is specifically interested in the first-order invariants, that is dynamical properties that do not depend on the choice of the model $M$.

Newelski proved early \cite{New3} that in stable theories, the Ellis group of the flow $S_G(M) = \SextX{G}{M}$ is isomorphic to the quotient $G/G^{00}$, which does not depend on the model. Moreover, the set of almost periodic types of $S_G(M)$ coincides with the set of \emph{generic types} types of $G$. This result was also obtained in the unstable setup of definably compact groups definable in the o-minimal setup. This has led to the conjecture that under some sufficiently ``tame'' conditions (inlcuding at least NIP), the Ellis group of $\SextX{G}{M}$ is isomorphic to $G/G^{00}$. This was proven true by Chernikov and Simon in \cite{CS} for the class of \emph{definably amenable} NIP groups. However, a counterexample was found by Gismatullin, Pillay and Yao in the o-minimal setting \cite{GPP}.

Among the difficulties outside the stable setup is that the almost periodic types no longer coincide with the generic types. Indeed, generic types may not even exist. A number of their generalizations were proposed. An early instance are \emph{weakly generic} types, which form the closure of almost periodic types. The notion of \emph{f-generic} and \emph{strongly f-generic} types is instrumental in the description of dynamics of definably amenable groups. Particular among those are groups with \emph{dfg}, that is definably amenable groups that admit a (strong) global f-generic type definable over some small model.

A substantial study over the Ellis group conjecture and surrounding subjects has been carried out in the context of algebraic groups over particular NIP fields. The original counterexample was the group $\SL_2(\mathbb{R})$, definable in an o-minimal expansion of the reals. The description of the Ellis group was later generalized to groups with compact-torsion-free decomposition \cite{Jag1}, and eventually all groups definable in the reals \cite{Yao1}, \cite{Jag2}. The work in \cite{PPYpadic} initiated similar study for groups definable in $\Qp$, finding $\SL_2(\Qp)$ to be another prototypical example. The study of groups in $\Qp$ also motivated the development of the notion of groups with \emph{dfg}, eventually giving their description in the p-adic case and showing that in such groups, almost periodics coincide with weakly generic types. The work by Kirk in \cite{Kirk} provided another counterexample, the group $\SL_2\left(\mathbb{C}((t))\right)$.

The particular results on $\SL_2$ serve as the motivation for first part of the paper (Sections 1 and 2). We consider the dynamical properties of an arbitrary NIP field $K$ with the final goal of obtaining at least a partial description of the dynamics of $\SL_2(K)$. We state the main result as follows, with $\Gm$ denoting the multiplicative group of the field interpreted in a sufficiently saturated extension:
\begin{theorem}
	\label{thm:counterexample}
	Let $T$ be a NIP theory that interprets a field $K$ with \mbox{$\GmT \neq \Gm$}. Then $T$ interprets a counterexample to the Ellis group conjecture, namely $\SL_2(K)$.
\end{theorem}
This result generalizes the conclusions of each of \cite{GPP}, \cite{PPYpadic} and \cite{Kirk} regarding the Ellis group conjecture. After proving the result, we also discuss some classes of fields for which the assumption on the multiplicative group is true.

The second part of the paper (Section 3) concerns algebraic groups with \emph{dfg}. These groups were studied in the $p$-adic context \cite{PYpadic}, and use dp-minimality of the field $\Qp$ in order to classify groups of dimension $1$ as either \emph{dfg} or with \emph{finitely satisfiable generics}. Here, we look at the general case of a dp-minimal field $K$ to consider the dynamics of its multiplicative group by analysing the work by Johnson in \cite{Johnson}. We identify families of global definable f-generics of the multiplicative group that witness \emph{dfg}. With some structural results, we extend this statement to unipotent groups, and assuming $\charf K = 0$, to trigonalizable groups.

\section{Preliminaries}
We will assume that the reader is familiar with the basics of model theory, including forking, definability and invariance of types, and the notions of Poizat's ``special sons'' (i.e. heirs and coheirs). Throughout the paper, we will often fix a first-order structure $M$ in a language $\LL$ and consider it together with a fixed, $\kappa$-saturated, $\kappa$-strongly homogeneous $\CC \succ M$, where $\kappa > |M| + |\LL|$ is a (strong) limit cardinal. We call such an extension a \emph{monster model} (this differs from the typical use of this term). We say that a set $A$ is \emph{small} (or \emph{bounded}) if $|A| < \kappa$. A definable set always means definable with parameters from $\CC$. We identify a definable set $X$ with the formula defining it, but also use $X$ as a shorthand for $X(\CC)$. In particular, ``$a \in X$'' means $\models X(a)$.

A type (over some parameters) is always complete unless stated otherwise, and a \emph{global} type means a complete type over $\CC$. If $N \subseteq \CC$ is a model and $p \in S(N)$ is an $A$-invariant type (with $A$ small), and $B$ is a set of parameters (possibly not small), then we write $p|B$ to denote the restriction to $B$ of the unique $A$-invariant extension of $p$ to a type over some large model containing $B$. A particular case is when $p \in S(M)$ is a definable type with its \emph{global heir} $p|\CC$.

\subsection{Definable topological dynamics}
The following is standard classical topological dynamics. Let $G$ be a (discrete) group. A $G$-\emph{flow} is a transitive action of $G$ on a compact Hausdorff topological space $X$ via homeomorphisms. A $G$-flow is called \emph{point transitive} if it contains a dense $G$-orbit. A set $Y \subseteq X$ is called a \emph{subflow} if it is nonempty, closed and invariant under the $G$-action. It is \emph{minimal} if it contains no proper subflows. For $g \in G$, we consider the associated homeomorphism $\pi_g \colon X \to X$. Equip the space $X^X$ with point convergence topology and function composition operation $*$. Let $E(X) = \cl\left( \{\pi_g \in X^X : g \in G\} \right)$. Then $E(X)$ is naturally a $G$-flow and $\left( E(X), * \right)$ is a semigroup, called the (enveloping) Ellis semigroup of the flow $X$. It is naturally isomorphic (topologically and algebraically) to its own Ellis semigroup. The dynamical properties of the action of $G$ on $X$ are studied via $E(X)$. Let $I \subseteq E(X)$ be a subflow. Then $I$ is minimal if and only if it is a (closed) ideal. Let $J(I)$ be the set of idempotent elements of $I$. Then we have the following decomposition of $I$ into a disjoint union:
\[I = \bigsqcup_{u \in J(I)} u * I.\]
Each set $u * I$ is a group that we call an \emph{ideal subgroup}. All ideal subgroups in $E(X)$ are isomorphic to each other (even for different $I$'s). Their isomorphism class is called the Ellis group of the flow $X$.

In the model-theoretic setup, we consider an arbitrary $\LL$-structure $M$ and an $M$-definable group $G$ acting definably and transitively on an $M$-definable set $X$. The group $G(M)$ acting on $\SextX{X}{M}$ is a point transitive $G(M)$-flow. In case where $G = X$, Ellis semigroup of $\SextX{G}{M}$ is naturally isomorphic (as a $G(M)$-flow) to itself, where the semigroup operation can be explicitly described. This is the \emph{universal definable flow} of $G$ over $M$.

Assuming NIP, we may consider $M$ in the language $\LL_{ext,M}$, the expansion of $\LL$ by predicates for all externally definable subsets of (Cartesian powers of) $M$. We denote this $\LL_{ext,M}$-structure by $M^{ext}$. By the classic result of Shelah, $\Th(M^{ext})$ has NIP and quantifier elimination, and all types over $M^{ext}$ are definable. Then one can identify the space $\Sext{M}$ with the space $S_{qf}(M^{ext})$ of quantifier-free types in $\LL_{ext,M}$, which by quantifier elimination is identified with the usual type space $S_G(M^{ext})$. Given an $M$-definable group $G$, this allows to pass to $M^{ext}$ and work with $S_{G}(M^{ext})$ as the universal definable $G$-flow over $M$. In this setup, the semigroup operation $*$ can be described as follows: for $q, p \in S_{G}(M^{ext})$, $q * p = \tp(a \cdot b/M^{ext})$ for $a \models q, b \models p|Ma$. Moreover, the groups $G^0, G^{00}, G^{000}$ do not change when passing from $\LL$ to $\LL_{ext,M}$.

\subsection{Connected components and definable amenability}
Let $M$ be a first order structure with a monster model $\CC \succ M$, $G$ an $M$-definable group, and $A \subset \CC$ a small set. Recall the following subgroups of $G = G(\CC)$, also know as the (model-theoretic) \emph{components} of $G$: $G_A^0$, the intersection of all $A$-definable subgroups of $G$ of finite index; $G_A^{00}$, the smallest $A$-type definable subgroup of $G$ with bounded index; and $G_A^{000}$, the smallest $A$-invariant subgroup of $G$ with bounded index. We have $G_A^{000} \leq G_A^{00} \leq G_A^0$ for all $G, A$. If $G_A^0$ does not depend on the set $A$, then we say that $G^0 := G_\emptyset^0$ exists, and similarly for the remaining components. Assuming NIP, each of $G^0, G^{00}$ and $G^{000}$ exist.

The original Ellis group conjecture by Newelski stated that under sufficiently ``tame'' conditions (essentially meaning NIP), the Ellis group of the universal definable flow of $G$ is isomorphic to $G/G^{00}$, making it independent of the model. It was proven true for stable groups and for definably compact groups definable in the o-minimal setup. A crucial result by Chernikov and Simon extends this to the class of definably amenable groups. We will give a brief outline. We say that a definable NIP group $G$ is definably amenable if there is a finitely additive probabilistic measure (a \emph{Keisler measure}) on the algebra of the definable subsets of $G$ invariant under the group action. Definable (NIP) groups that are definably amenable include all stable and all solvable groups (more generally, groups amenable in the classical sense as discrete groups).

The dynamics of $G$ is then explained through its f-generic types. Their complete definition can be found in \cite{CS}, but in this paper, we will use the following characterization instead:
\begin{fact}[\cite{CS}, Proposition 3.8]
	Let $G$ be a definably amenable NIP group and $p \in S_G(\CC)$. The following are equivalent:
	\begin{enumerate}[(i)]
		\item p is f-generic.
		\item $\Stab(p) = G^{00}$.
		\item The $G$-orbit of $p$ is bounded.
	\end{enumerate}
\end{fact}
A (global) f-generic type $p$ is called \emph{strongly} f-generic if it is invariant over some small model. By collecting some of the core results of \cite{CS}:
\begin{fact}
	\label{fact:cs fundamental}
	Let $G$ be definably amenable NIP group and $M \subset \CC$ a small model. Then:
	\begin{enumerate}[(i)]
		\item $G$ admits a (strong) f-generic type.
		\item An almost periodic $p \in S_G(\CC)$ is f-generic.
		\item If all types over $M$ are definable and $p \in S_G(M)$ is almost periodic, then $p = p' \restriction M$ for some f-generic $p' \in S_G(\CC)$.
		\item (Ellis group conjecture) If $I \subset S_G(M^{ext})$ is a minimal flow with an idempotent $u \in I$, then the map
		\[u * I \ni p \mapsto p/G^{00} \in G/G^{00}\]
		is a group isomorphism.
		\item $G^{000} = G^{00}$.
	\end{enumerate}
\end{fact}
A definably amenable NIP group $G$ where $G^{00} = G$ is called \emph{definably extremely amenable}. In such case, every f-generic type of $G$ is a fixed point of the group action, and hence almost periodic.

\begin{remark}
	As pointed out in \cite{KPbohr}, the Ellis group of the universal definable flow of $G$ naturally maps onto $G/G^{000}$, so the modernized statement of the Ellis group conjecture should involve the quotient $G/G^{000}$ in place of $G/G^{00}$. For definably amenable groups, this statement is equivalent to the original one by Fact \ref{fact:cs fundamental}(v).
\end{remark}

\subsection{Liftings of definable functions}
Later in the paper, we will often canonically extend various field [group] operations on definable fields [groups] to operations on certain spaces of types. We now make this notion precise, providing somewhat more context than necessary. The following notation is taken from \cite{HP}:
\begin{definition}
	Let $p, q$ be invariant global types. Then $p \otimes q := \tp(a,b/\CC)$, where $a \models p$ and $b \models q|\CC a$.
\end{definition}

Note that the definition in \cite{HP} uses the opposite order, namely $b \models q$ and $a \models p|\CC b$; we change the order for convenience.
For a small set $A \subset \CC$, write
\[\Sinv{A} := \{p \in S(\CC): p \text{ is }A\text{-invariant}\},\]
and $\SinvX{X}{A} = \Sinv{A} \cap [X]$ for a definable $X$. We note some basic observations about $\otimes$:

\begin{fact}[\cite{HP}]
	\begin{enumerate}[(i)]
		\item If $p,q \in \Sinv{A}$, then also $p \otimes q \in \Sinv{A}$.
		\item The operation $\otimes$ is associative.
	\end{enumerate}
\end{fact}
As $\otimes$ is associative, we may unambiguously write $p_1 \otimes \ldots \otimes p_n$ for any global invariant types $p_1, \ldots, p_n$. 

We will need the following:
\begin{fact}
	\label{fact:prodDef}
	If $p$ and $q$ are definable over $A$, then $p \otimes q$ is also definable over $A$.
\end{fact}
\begin{proof}
	Consider $p, q$ as types in variables $x$ and $y$ respectively and let $a \models p$. Let $\phi(x, y, z)$ be a formula without parameters. Since $q$ is $A$-definable, there is a formula $\psi(x, z)$ over $A$ such that for all $c \in \CC$ we have
	\[\phi(a, y, c) \in q|\CC a \iff \models \psi(a, c) \iff \psi(x, c) \in p.\]
	But by $A$-definability of $p$, $\psi(x, c) \in p \iff \models \chi(c)$ for some formula $\chi(z)$ over $A$. Then $\chi(z)$ is a $\phi$-definition of $p \otimes q$.
\end{proof}

Note that if for some element $a$ the type $\tp(a/\CC)$ is $A$-invariant [$A$-definable] and $f$ is an $A$-definable function, then $\tp(f(a)/\CC)$ is $A$-invariant [respectively $A$-definable]. This allows us to write the following definition:
\begin{definition}
	\label{def:lifting}
	Let $M$ be a small model. Suppose that $Y$, $X_i$ and $f \colon X_1 \times \ldots \times X_n \to Y$ are all $M$-invariant. The \emph{lifting of $f$ over $M$} is a function $\tilde{f} \colon \SinvX{X_1}{M} \times \ldots \times \SinvX{X_n}{M} \to \SinvX{Y}{M}$ defined as follows:
	\[\tilde{f}(p_1, \ldots, p_n) = \tp(f(a_1, \ldots, a_n)/\CC),\]
	where $a_1, \ldots, a_n \models p_1 \otimes \ldots \otimes p_n$.
\end{definition}

We can also use the notation $\otimes$ outside of the global context. Let $M \subset \CC$ be a small model. Consider the set $\Sdef{M} \subseteq S(\CC)$ of global types definable over $M$. This set is not necessarily closed in $S(\CC)$ (contrast this with the subspace $\Sinv{M}$ of all $M$-invariant global types, and the subspace $\Sfin{M}$ of all global types finitely satisfiable in $M$, which are both closed in $S(\CC)$).
We can equip the set $\Sdef{M}$ with \emph{$M$-topology} by declaring the basic open sets to be of the form $\Sdef{M} \cap [\phi]$, where $\phi$ is a formula over $M$. This topology is compact and is a coarsening of the natural subspace topology inherited from $S(\CC)$ (this inherited topology is usually richer, for instance discrete).

In the paper, we will often work the space $S(M)$ of types over a model $M$ such that all types $p \in S(M)$ are $M$-definable. For convenience, we will identify those types with  global types in the following way. For each $p \in S(M)$, its global heir $p|\CC$ is the unique $M$-definable global type extending $p$. Moreover, every $M$-definable global type is the heir of its restriction to $M$. Hence $p \mapsto p|\CC$ is a natural bijection between $S(M)$ and $\Sdef{M}$. This map is clearly open. If $\Sdef{M}$ is considered with the $M$-topology described in the paragraph above, the map is also a homeomorphism. Thus we can identify $S(M)$ with $\Sdef{M}$ with the $M$-topology as topological spaces. Since $\Sdef{M} \subset \Sinv{M}$, we can consider $\otimes$ as an operation on $S(M)$. For $p, q \in S(M)$, we have $p \otimes q = \tp(a, b/M)$ for $a \models p, b \models q|Ma$.

In this paper, lifting of functions will appear in the following contexts. First, we will often consider a field $K$ (possibly with additional structure), allowing us to take the lifting of multiplication, denoted as $\odot$. Secondly, given a group $G$ acting definably on a set $X$ (all definable in a first order structure $M$), we will lift the group action $\cdot \colon G \times X \to X$ to a function $* \colon \SinvX{G}{M} \times \SinvX{X}{M} \to \SinvX{X}{M}$.

In case of a structure $M$ over which all types are definable and an $M$-definable group $G$, the lifting $* \colon S_G(M) \times S_G(M) \to S_G(M)$ of the group operation on $G$ is precisely the semigroup operation of the Ellis semigroup $S_G(M)$. Similarly, if $G$ acts $M$-definably on an $M$-definable set $X$, the lifting of this action is the semigroup action of $S_G(M)$ on $S_X(M)$. In case of a field $K = M$, lifting its multiplication yields the Ellis semigroup $(\Gm(K), \odot)$.

\section{Fields and some linear group actions}
We work with a fixed infinite NIP field $K$ along with a fixed monster model \mbox{$\CC \succ K$}. In the following two subsections we will consider definable flows associated with the following $\emptyset$-definable transitive linear group actions:
\begin{compactitem}
	\item The additive group $\Ga = (\CC, +)$ and the multiplicative group $\Gm = (\CC^\times, \cdot)$, each acting on itself.
	\item The group $\SL_2$ acting on the projective line $\PP_1$ via M\"obius transformations; and $\SL_2$ acting on itself.
\end{compactitem}
The dynamics of $\SL_2(-)$ has been explicitly studied over three different NIP fields: $\mathbb{R}$ in \cite{GPP}, $\Qp$ in \cite{PPYpadic}, and $\mathbb{C}((t))$ in \cite{Kirk}. These analyses were carried out using different decompositions of $\SL_2$, often specific to the underlying field; and started with the study of the dynamical properties of the field itself. We can observe that in each case, the Ellis group of the universal definable flow of $\SL_2$ was shown to be isomorphic to the Ellis group of the universal flow of the multiplicative group of the underlying field. In this section, we analyze $\SL_2$ using tools of basic linear algebra that do not depend on the field. While we do not show the mentioned isomorphism, we are able to show that $\SL_2(K)$ is a counterexample to the Ellis group conecture if the Ellis group of the universal definable flow of the multiplicative group of $K$ is not trivial. 

\subsection{Additive and multiplicative groups}
First we consider the dynamics of $\Ga$ acting on itself which we treat as $\Ga$ acting on $\CC$ via translations, and similarly for $\Gm$ acting on $\CC^\times$. Both $\Ga$ and $\Gm$ are abelian, hence definably amenable. We will discuss then in terms of their global (strong) f-generic types.
The case of the additive group turns out to be trivial. Without qualification, we have:
\begin{lemma}
	\label{lem:additive is dea}
	The group $\Ga$ is definably extremely amenable.
\end{lemma}
\begin{proof}
	Let $p$ be a global f-generic of $\Ga$. Then $p$ is stabilized by $\GaT$, but $\GaT = \Ga$ for any infinite field.
\end{proof}

The analogous result holds for $\Gm$ if $K$ is stable (where the unique generic of $K$ is multiplicatively stabilized by $\GmT = \Gm$), but not in the general NIP case.

We now construct a type in $S(K)$ which is simultaneously almost periodic with respect to the actions of $\Ga$ and $\Gm$. In the lemma below, note that an invariant global type type is a strong additive f-generic if and only if it is invariant under addition.
\begin{lemma}
	Suppose that $p \in S_{\Ga}(\CC)$ is a strong additive f-generic.
	\label{lem:good type}
	\begin{enumerate}[(i)]
		\item Let $q \in S_{\Gm}(\CC)$ be invariant. Suppose $a, b \in \bar \CC \succ \CC$ such that $b \models q$, and let $x \models p|\CC ab$. Then $a+bx \models q \odot p$.
		\item Let $q \in S_{\Gm}(\CC)$ with $q$ invariant. Then $q \odot p$ is a strong additive f-generic.
	\end{enumerate}
\end{lemma}
\begin{proof}
	(i) We have $a + bx = b(b^{-1}a + x)$. Since $x \models p|\CC ab$, also $b^{-1}a + x \models p|\CC ab$ and so $b(b^{-1}a + x) \models \tp(b/ \CC) \odot p = q \odot p$.
	(ii) Let $(x, y) \models q \otimes p$ and $c \in \CC$. Then by (i), $c + xy \models p \odot q$.
\end{proof}
\begin{corollary}
	Let $q\in S_{\Gm}(\CC), p \in S_{\Ga}(\CC)$ be global strong f-generics in the sense of $\Gm$ and $\Gm$ respectively. Then $q * p$ is a global strong f-generics with respect to both group actions.
\end{corollary}
\begin{proof}
	By Lemma \ref{lem:good type}(ii), $q * p$ is a strong f-generic of $\Ga$. It is a strong f-generic of $\Gm$ since $S_{\Gm}(\CC) \cdot q * p = (S_{\Gm}(\CC) \cdot q) * p$.
\end{proof}

\begin{remark}
	Note that the group $\Aff = \Gm \ltimes \Ga$ of affine transformations of the field is solvable, hence definably amenable. A strong f-generic $p$ of $\Aff$ is stabilized by $\Aff^{00} = \GmT \ltimes \GaT$. $\Aff$ acts on $X = \CC$ transitively, where the action restricted to $\Gm$ and $\Ga$ is the action by multiplication and by addition, respectively. We may identify $X$ with $\Aff/H$, where $H = \Stab(x)$ for any $x \in \CC$. Then $p/H$ is invariant and stabilized by both $\GmT$ and $\GaT$, hence a (strong) f-generic in the sense of both addition and multiplication.
\end{remark}

\subsection{M\"obius transformations and $\SL_2$}

Recall that the group $\GL_2$ acts transitively on the projective line $\PP_1$ via M\"obius transformations. We identify the projective line $\PP_1$ with the field extended by a point at infinity, namely $\PP_1 = \CC \cup \{\infty\}$. Then, the action is defined as follows:
\[\QM{a}{b}{c}{d} \cdot_1 x = \begin{cases} \frac{ax+b}{cx+d} &\mbox{if } x \neq \infty, \\
\frac{a}{c} & \mbox{otherwise}. \end{cases}
\]
Above, should the denominator of the right-hand side (in either case) be equal to $0$, the result is interpreted as $\infty$. The action can be restricted to the group $\SL_2 < \GL_2$. This restricted action is still transitive.

Furthermore, consider the subgroup $H < \SL_2$ consisting of matrices of the form $\QM{1}{a}{0}{1}$ with $a \in \CC$. Then $H$ is naturally isomorphic to the group $\Ga$. The action $\cdot_1$ restricted to $H$ has two orbits: $\CC$ and $\{\infty\}$. The action of $H$ on $\CC$, under the identification $\Ga \cong H$, is just $\Ga$ acting on $\CC$ by translations.

We now consider the flow over the field $K$. The action $\cdot_1$ of $\SL_2(K)$ on $\PP_1(K)$ extends to the flow $(\SL_2(K), S_{\PP_1}(K))$ and can also be lifted to the semigroup action $*_1$ of $S_{\SL_2}(K)$ on $S_{\PP_1}(K)$.

With the notation as above:

\begin{lemma}
	\label{lem:mobius add invariant}
	Suppose that $I \subset S_{\PP_1}(K)$ is a minimal subflow in the sense of the action of $\SL_2(K)$ by M\"obius transformations. Then $I \subset S(K)$ and there is $p \in I$ invariant under addition.
\end{lemma}
\begin{proof}
	We have $S_{\PP_1}(K) = S(K) \cup \{\tp(\infty/K)\}$ and since $\tp(\infty/K)$ is algebraic, it is omitted by $I$. Hence $I$ is a closed subset of $S(K)$ that is invariant under the action of $\SL_2(K)$, and therefore under the action of $H(K)$ (that is, invariant under addition). Thus $I$ can be considered as a subflow of $(\Ga(K), S(K))$ and hence it contains an almost periodic in the sense of addition.
\end{proof}

Now we move to the analysis of the dynamics of the group $\SL_2$. Recall that for any field $K$, $\SL_2(K)$ does not have infinite, proper normal subgroups. Hence for any infinite NIP field, $\SL_2^{000} =\SL_2$. So for any such field, the conjecture asserts that the Ellis group of the universal flow of $\SL_2$ over $K$ is trivial. We will show that this is the case only if $\GmT = \Gm$. Our main statement is rather modest:
\begin{proposition}
	\label{prop:size of ellis group}
	Suppose that $K$ is a NIP field with all types over $K$ definable. Then the Ellis group of the universal flow of $\SL_2$ has the size at least $\Gm/\GmT$.
\end{proposition}
As a corollary, we obtain Theorem \ref{thm:counterexample}

We will take certain shortcuts in proving the proposition. We will not give a precise description of (some) minimal subflow of the universal flow of $\SL_2$ over $K$, but an approximation sufficient to determine the size of its ideal subgroups. We also do not describe the group structure of the ideal subgroups. At the end of the section we will discuss certain cases of fields with $\GmT \lneq \Gm$, specifically for valued fields and dp-minimal fields.

We begin with a general fact about almost periodic types of an algebraic group. Let $K$ be any NIP field (possibly with additional structure) and $\CC \succ K$ a monster model. Suppose that $G$ is an algebraic group over $K$. For any $x \in G$ we define $\dim_K(x)$ as the \textit{algebraic (Zariski) dimension} of $x$ over $K$, namely the minimum algebraic dimension over $K$ of an algebraic (over $K$) variety $X \subseteq G$ containing $x$. Then $\dim_K(x)$ depends only on $\tp(x/K)$. Hence for $p \in S_G(K)$ we define $\dim_K(p) = \dim_K(x)$ for any $x \models p$. 
Now assume that all types over $K$ are definable. With the notation as above:

\begin{lemma}
	\label{lem:zariski dimension}
	For any $q, p \in S_G(K)$, we have $\dim_K(p) \leq \dim_K(q * p)$.
\end{lemma}
\begin{proof}
	Let $(g, h) \models q \otimes p$. Suppose that $\dim_K(q * p) \leq n$. Then there is a $K$-definable subvariety $X \subseteq G$ of dimension $n$ such that $g \cdot h \in X$. Since $\tp(g/Kh)$ is finitely satisfiable in $K$, there is some \mbox{$g_0 \in G(K)$} such that $g_0 \cdot h \in X$, i.e. $h \in g_0^{-1} \cdot X$. Hence $\dim_K(h) \leq n$ since the map \mbox{$x \mapsto g_0^{-1}x$} is birational.
\end{proof}
Furthermore:
\begin{corollary}
	\label{cor:max dim aper}
	The set
	\[J = \{p \in S_G(K) : \dim_K(p) = \dim_K(G)\}\]
	is a subflow of $S_G(K)$. In particular, there is an almost periodic type $p \in S_G(K)$ with $\dim_K(p) = \dim_K(G)$.
\end{corollary}

We now consider $G = \SL_2$ and its universal definable flow $(G(K), S_G(K))$. Let $p \in S_G(K)$ be a fixed almost periodic type given by Corollary $\ref{cor:max dim aper}$. Let $I = S_G(K) * p$ be the minimal subflow of $S_G(K)$ containing $p$. Then for every $p' \in I$ and every matrix $A \models p'$, we have $\dim_K(A) = \dim_K(\SL_2) = 3$.

Recall that a upper-unitriangular matrix is one of the form $\QMs{1}{x}{0}{1}$ for $x \in \Ga$, and an (invertible) lower-triangular one has the form $\QMs{\alpha}{0}{\gamma}{\delta}$ with $\gamma \in \Ga, \alpha, \delta \in \Gm$. Let us explicitly recall the following matrix decomposition (via Gaussian elimination):
\begin{fact}
	Let $A = \QM{a}{b}{c}{d} \in \GL_2$ such that $d \neq 0$. Then $A$ admits a unique decomposition into a product of a upper-unitriangular and lower-triangular matrices:
	\[\QM{a}{b}{c}{d} = \QM{1}{bd^{-1}}{0}{1} \QM{a - bcd^{-1}}{0}{c}{d}.\]
\end{fact}
Any matrix admitting this decomposition can be (uniquely) written as $\QMs{1}{\zeta}{0}{1}\QMs{\delta^{-1}}{0}{\gamma}{\delta}$ with $\zeta, \gamma \in \Ga, \delta \in \Gm$. It is easy to see that if $A = \QMs{a}{b}{c}{d}$ has dimension $3$ over $K$, then $a,b,c,d \neq 0$. Consequently, for every $p' \in I$ and $A \models p'$, $A$ admits the decomposition.

We now consider the action of $S_G(K)$ on the almost periodic $p$.
\begin{lemma}
	\label{lem:good aper}
	There is an almost periodic $p' \in I$ realized by $\QMs{1}{\zeta'}{0}{1}\QMs{\delta'^{-1}}{0}{\gamma'}{\delta'}$ such that $\tp(\zeta'/K)$ is invariant under addition.
\end{lemma}
\begin{proof}
	Consider $q \in S_G(K)$. Let $A = \QMs{a}{b}{c}{d} \models q$ and $\QMs{1}{\zeta}{0}{1}\QMs{\delta^{-1}}{0}{\gamma}{\delta} \models p|Kabcd$. Then by a direct calculation:
		\begin{align*}
		\QM{a}{b}{c}{d} \cdot \QM{1}{\zeta}{0}{1} \cdot \QM{\delta^{-1}}{0}{\gamma}{\delta} & =
		\QM{1}{A \cdot_1 \zeta}{0}{1} \cdot \QM{(c\zeta+d)^{-1}}{0}{c}{c\zeta+d} \cdot \QM{\delta^{-1}}{0}{\gamma}{\delta}\\
		& = \QM{1}{A \cdot_1 \zeta}{0}{1} \cdot \QM{(c\zeta+d)^{-1}\delta^{-1}}{0}{\frac{c}{\delta}+(c\zeta+d) \gamma}{(c\zeta+d)\delta},
		\end{align*}
	where $\cdot_1$ is the action on $\Ga$ by the M\"obius transformations. The right hand side of the calculation is a realization of $q * p$. Its upper-unitriangular component has the coefficient $A \cdot_1 \zeta$, which realizes $q *_1 \tp(\zeta/K)$. As $q$ ranges over all $S_G(K)$, the set of types obtained this way is a flow in the sense of the action by M\"obius transformations. By Lemma \ref{lem:mobius add invariant}, such a flow contains a type invariant under addition. The conclusion follows.
\end{proof}

We will now assume that $p$ is such as stipulated by Lemma \ref{lem:good aper}.

\begin{fact}
	The set $p * I = p * S_G(K) * p$ is an ideal subgroup containing $p$.
\end{fact}
\begin{proof}
	The follows from the general theory of compact semigroups. There is an idempotent $u \in I$ such that $p$ is the element of the group $u * I$ in which $u$ is the identity. Then $p * I = p * u * I = u * I$.
\end{proof}

We can now prove the main statement.
\begin{proof}[Proof of Proposition \ref{prop:size of ellis group}]
Consider the group $D < \SL_2$ of diagonal matrices. $D$ can be identified with $\Gm$ by the map $\Gm \ni a \mapsto \QMs{a}{0}{0}{a^{-1}} \in D$. The ideal subgroup $p * S_G(K) * p$ contains the subset $p * S_D(K) * p$. Define the map $\pi\colon I \to \Gm/\GmT$ the following way. If $p' \in I$ is realized by $\QMs{1}{\zeta'}{0}{1}\QMs{\delta'^{-1}}{0}{\gamma'}{\delta'}$, then $\pi(p') = \delta'/\GmT$. We will show that $\pi$ restricted to $p * S_D(K) * p$ is onto $\Gm/\GmT$.

Fix a realization $\QMs{1}{\zeta}{0}{1}\QMs{\delta^{-1}}{0}{\gamma}{\delta} \models p$. The type $\tp(\zeta/K)$ is invariant under addition. For $a \in \Gm$, write 
\[p_a = p * \tp\big(\QMs{a}{0}{0}{a^{-1}}/K\big) * p \in p * S_D(K) * p.\]

\begin{claim2}
	For all $a \in \Gm$, $\pi(p_a) = a/\GmT \cdot (\zeta\gamma\delta)/\GmT$.
\end{claim2}
\begin{innerproof}[Proof of Claim]
	Without loss of generality assume that $a$ is such that $\tp(a/K\zeta\gamma\delta)$ is the heir of $\tp(a/K)$ and take $\QMs{1}{\zeta'}{0}{1}\QMs{\delta'^{-1}}{0}{\gamma'}{\delta'} \models p|\zeta\gamma\delta a$. Then $p_a$ is realized by
	\begin{align*}
	\QM{1}{\zeta}{0}{1}\QM{\delta^{-1}}{0}{\gamma}{\delta} \cdot \QM{a}{0}{0}{a^{-1}} \cdot \QM{1}{\zeta'}{0}{1}\QM{\delta'^{-1}}{0}{\gamma'}{\delta'} = \\
	= \QM{1}{\frac{a \zeta' \left(\frac{1}{\delta }+\gamma  \zeta\right)+\frac{\delta  \zeta}{a}}{\frac{\delta }{a}+a \gamma  x'}}{0}{1} \cdot \QM{\frac{1}{\delta ' \left(\frac{\delta }{a}+a \gamma  \zeta'\right)}}{0}{\frac{a \gamma }{\delta '}+\gamma ' \left(\frac{\delta }{a}+a \gamma  \zeta'\right)}{\delta ' \left(\frac{\delta }{a}+a \gamma  \zeta'\right)}.
	\end{align*}
	Hence $\pi(p_a) = \left(\delta'\left(\frac{\delta }{a}+a \gamma  \zeta'\right)\right)/\GmT = \left(\frac{\delta }{a}+a \gamma  \zeta'\right)/\GmT \cdot \delta / \GmT$. Since $\tp(\zeta'/K\gamma\delta a)$ is the heir of $\tp(\zeta/K)$, we have $\frac{\delta }{a}+a \gamma \zeta' \equiv_K a \gamma \zeta'$. Consequently, $\pi(p_a) = \left(a \gamma  \zeta'\right)/\GmT \cdot \delta / \GmT = a/\GmT \cdot \left(\zeta\gamma\delta\right)/\GmT$ as claimed. \qedhere
\end{innerproof}
We finish by nothing that the coset $(\zeta\gamma\delta)/\GmT$ is fixed and determined by $p$. Hence as $a$ ranges over all cosets of $\GmT$, so does $\pi(p_a)$.
\end{proof}

We finish the section with a brief look at some cases when $\GmT$ is strictly smaller than $\Gm$. Suppose that $K$ is a valued field, i.e. it is equipped by some definable valuation $v \colon \Gm \to \Gamma$ with $\Gamma$ the value group. We implicitly assume $\Gamma$ to be the image of $v$. The preimage $v^{-1}[\Gamma^{00}]$ is an invariant subgroup of $\Gm$ with bounded index. Hence clearly:
\begin{fact}
	If $\GmT = \Gm$, then $\Gamma^{00} = \Gamma$.
\end{fact}

\begin{corollary}
	If $K$ is field with a discrete valuation $v \colon \Gm \to \Gamma$, then $\GmT \neq \Gm$.
\end{corollary}
\begin{proof}
	Clearly, $v^{-1}[2 \Gamma]$ is a proper subgroup of $\Gm$ with finite index.
\end{proof}
Note that the above case includes the $p$-adics and the field $\mathbb{C}((t))$.

We also consider dp-minimal fields, which we will study in more detail in the next section. For now, we recall:
\begin{fact}[\cite{Johnson}]
	If $K$ is dp-minimal, then $[\Gm:\mathbb{G}^n_m] < \infty$ for any $n > 0$.
\end{fact}

From this we immediately get:
\begin{corollary}
	Suppose that $K$ is dp-minimal and $\GmD = \Gm$. Then:
	\begin{enumerate}[(i)]
		\item $\Gm$ is divisible.
		\item If $K$ is a valued field with the value group $\Gamma$, then $\Gamma$ is divisible.
	\end{enumerate}
\end{corollary}

\section{Definable $f$-generics}
The following definition is due to Pillay and Yao:
\begin{definition}
	A definably amenable group $G$ has \emph{dfg} (\emph{definable f-generics}) if it admits a global definable f-generic type.
\end{definition}

Usefulness of the existence of a definable global f-generic can be motivated by the following result:
\begin{theorem}[\cite{YaoTri}, Lemma 2.3]
	Let $G$ be \emph{dfg} with a global definable f-generic $p$. Then the orbit of $p$ is closed, and $p$ is almost periodic.
\end{theorem}
As a corollary, the orbit of $p$ consists of precisely one ideal subgroup.

The paper \cite{YaoTri} dealt with the question whether f-generic global types coincide with almost periodic types in the setup of definably amenable linear groups definable the field $\Qp$ of $p$-adic numbers. The results are done in the context of previous analyses by Pillay and Yao of groups definable in $\Qp$, particularly \emph{dfg} groups; and the work to classify groups of dimension $1$. The latter analysis makes use of the fact that groups of dimension $1$ are dp-minimal, and shows that dp-minimal groups have either \emph{dfg} or \emph{finitely satisfiable generics}.

In this section, we show that dp-minimality of a field $K$ is enough to show that various algebraic groups over $K$ are \emph{dfg}. 
Specifically, in the following subsection we show a number of basic structural results about definable (in arbitrary structure) groups with \emph{dfg}. In the next subsection, we prove:
\begin{theorem}
	\label{thm:dp minimal field}
	Let $K$ be a dp-minimal field such that all types over $K$ are definable. Then $\Gm$ has \emph{dfg}, witnessed by a definable global infinitesimal type.
\end{theorem}

Along with the structural results, the theorem shows that, in its setup, groups obtained from $\Ga$ and $\Gm$ via constructions such as semidirect products or group extensions are also \emph{dfg}. We will provide certain examples after the theorem is proved.

\subsection{Structural results}
In this subsection, we assume that $M$ is an arbitrary NIP structure with a monster model $\CC \succ M$. We will need the following:
\begin{lemma}
	\label{lem:definable extension}
	Assume that all types over $M$ are definable. Then any global $\Delta$-type $\pi$ definable over $M$ extends to a complete global type definable over $M$.
\end{lemma}
\begin{proof}
	Let $\pi' = \pi \restriction M$ and take any extension $\pi' \subseteq p' \in S(M)$. Then $p'$ is definable and its global heir extends $\pi$.
\end{proof}

The following is essentially Remark 2.20 from \cite{PYpadic}:
\begin{fact}
	Suppose that $G, H$ are $M$-definable groups and $f \colon G \to H$ is $M$-definable and onto. If $G$ has \emph{dfg}, then $H$ has \emph{dfg}.
\end{fact}

The following two results are analogous to similar ones for (classical and definable) amenable groups, with adapted proofs:

\begin{lemma}
	Suppose that $G$ is \emph{dfg} and $H \leq G$ with $G, H$ $M$-definable. If $G/H$ has a definable section, then $H$ is \emph{dfg}.
\end{lemma}
\begin{proof}
	Let $\pi \colon G \to G/H$ be the quotient map. Let $X$ be a definable section of $G/H$ and $p \in S_G(\CC)$ a definable f-generic. Define a set of formulas $r$ such that for $\phi \in \LL(\CC)$:
	\[\phi(x) \in r \iff (\phi(x) \cap H) \cdot X \in p.\]
	For any formula, we have $\left((\phi(x) \cap H) \cdot X\right)^c = (\lnot \phi(x) \cap H) \cdot X$, hence $r$ is a consistent and complete type in $S_H(M)$. It is clearly definable. It is also stabilized by $N^{00}$: consider $n \in N^{00} \subseteq G^{00}$. Then:
	\begin{align*}
	\phi(x) \in r \Rightarrow \left(\phi(x) \cap H\right) \cdot X \in p & \Rightarrow n\left(\left(\phi(x) \cap H\right) \cdot X\right) \in p \\
	& \Rightarrow \left(n\phi(x) \cdot X\right) \in p \Rightarrow n\phi(x) \in r.\qedhere
	\end{align*}
\end{proof}
Conversely:
\begin{proposition}
	\label{prop:dfg extension}
	Assume that all types over $M$ are definable. Suppose that $G$ is definably amenable and $N \lhd G$ with $G, N$ $M$-definable. If both $N$ and $G/N$ are \emph{dfg}, then $G$ is \emph{dfg}.
\end{proposition}
\begin{proof}
	Write $H = G/N$. Let $p \in S_N(\CC)$ and $q' \in S_H(\CC)$ be definable f-generic. Let $q \in S_G(\CC)$ be a definable extension of the partial type $\pi^{-1}[q']$ given by Lemma \ref{lem:definable extension}. Let $r = q \odot p$. Then $r$ is a definable type in $S_G(\CC)$. We will show that it has a bounded orbit. It is sufficient to show that it has boundedly many translates by the elements of $\pi^{-1}[H^{00}]$, a subgroup of $G$ with bounded index. Let $I = S_N(\CC) * p$ be the minimal flow generated by $p$. It is exactly the (closed) $N(\CC)$-orbit of $p$.
	
\begin{claim2}
	Let $s \in S_G(\CC)$. Then the set $s * I$ is determined by $\pi(s) \in S_H(\CC)$.
\end{claim2}
\begin{innerproof}[Proof of Claim]
	Take any $s, s' \in S_G(\CC)$ such that $\pi(s) = \pi(s')$. It is sufficient to show that $s' * I \subseteq s * I$. There are realizations $g \models s, g' \models p'$ such that $\pi(g) = \pi(g')$, that is $g' = gn$ for some $n \in N$. Let $p_0 \in I$ and $h \models p_0|\CC gn$. Then $gn h \models s' * p_0$. On the other hand, since the orbit of $p_0$ is closed, $\tp(nh/\CC g)$ is the heir of some type $p_1 \in I$ and hence
	\[gnh \models \tp(g/\CC) * \tp(nh/\CC) = s * \tp(n/\CC) * p \in s * I. \qedhere \]
\end{innerproof}
To complete the argument, take any $g \in \pi^{-1}[H^{00}]$ and consider $g \cdot q \odot p$. Since $\pi(g) \in H^{00} = \Stab_H(q')$, we have $\pi(g \cdot q) = \pi(q)$ and by the Claim, $g \cdot q \odot p = (g \cdot q) \odot p \in q * I$. Finally, $q * I$ is of bounded size since $|q * I| \leq |I|$ and $I$ is the (bounded) orbit of $p$.
\end{proof}

\subsection{Dp-minimal fields and \emph{dfg}}
In this subsection we perform an abridged reading of Sections 3 and 4 of \cite{Johnson} from the perspective of definable topological dynamics. The main goal is to recall the results concerning the existence of a canonical topology on a given dp-minimal field and use them to prove Theorem \ref{thm:dp minimal field}.

We will start with a brief overview. A theory (or a structure) is called \emph{dp-minimal} if it has dp-rank $0$ or $1$. Let $K$ be a dp-minimal, unstable field and $\CC \succ M$ a monster model. The work in \cite{Johnson} establishes the existence of a canonical topology on $K$ by constructing a partial type $I_K$ over $K$ and declaring that $\tau_K = \{\phi(K) : \phi \in I_K\}$ is the basis of open neighborhoods of $0$. This topology is shown to be consistent with the field operations. Moreover, this topology arises from a certain Henselian valuation on $\CC$: there is a $\bigvee$-definable ring $\OK \subseteq \CC$ containing $K$ that is valuational, with it smaximal ideal equal to $I_K(\CC)$. This ideal consists of all elements of $\CC$ contained in every $K$-definable (open) neighborhood of $0$; that is, elements that are \emph{infinitesimal} over $K$.

We will now recall the precise definitions and results that we need on the type $I_K$ and the valuational ring $\OK$ of $\CC$.

\begin{fact}[\cite{Johnson}, Observation 2.2]
	$\Th(K)$ eliminates $\exists^{\infty}.$
\end{fact}
In fact, for a formula $\phi(x,\bar{y})$ there is a formula $t_{\phi}(\bar{y})$ without parameters such that for all $\bar{b}$ we have $\models t_\phi(\bar{b}) \iff \models \exists^{\infty}x \phi(x,\bar{b})$.

Let $X, Y \subseteq \CC$. Define
\[X -_\infty Y := \{c \in \CC : \exists^{\infty} y \in Y ~ (c + y \in X)\}.\]
If $X, Y$ are definable over some set of parameters $A$, then $X -_\infty Y$ is definable over $A$.

With this operation, Johnson defines the following partial type:
\[I_K := \{X -_\infty X : X \text{ is infinite and } K\text{-definable}\}.\]

A complete type $p \in S(K)$ is called \emph{infinitesimal (over $K$)} if $p$ extends $I_K$. An element of $\CC$ realizing an infinitesimal type over $K$ is called \emph{infinitesimal over $K$}. By definition, an element $c \in \CC$ is infinitesimal over $K$ if and only if for every infinite, $K$-definable set $X$, the set $X \cap (c + X)$ is also infinite.

\begin{fact}[\cite{Johnson}]
	\phantomsection \label{fact:infClosed}
	\begin{enumerate}[(i)]
		\item $I_K(\CC)$ is a subgroup of $(\CC, +)$.
		\item $I_K(\CC)$ is invariant under multiplication by $K$; that is, $K \cdot I_K(\CC) = I_K(\CC)$.
	\end{enumerate}

\end{fact}

The last step is to define the ring $\OK \subseteq \CC$. It is precisely the multiplicative stabilizer of $I_K(\CC)$ in $\CC$:
\[\OK := \{x \in \CC : x \cdot I_K(\CC) \subseteq I_K(\CC)\}.\]

\begin{fact}[\cite{Johnson}]
	$\OK$ is a valuational subring of $\CC$ with the maximal ideal $I_K(\CC)$.
\end{fact}

The valuational ring induces a valuation $v_K \colon \CC^\times \to \CC^\times/\OK^{\times}$.

\begin{fact}[\cite{Johnson}]
	\begin{enumerate}[(i)]
		\item $K \subseteq \OK$. Hence, $v_K$ is trivial on $K$.
		\item $\OK$ is $\bigvee$-definable.
		\item The valuation $v_K$ is Henselian.
	\end{enumerate}
\end{fact}
\begin{remark}
	\begin{enumerate}[(i)]
		\item As remarked in \cite{Johnson}, the notation here is somewhat unfortunate: $\OK$ is not a valuational ring of $K$; the subscript refers to the (small) set of parameters $K$.
		\item Regarding item (ii), it is easy to note that if $x \neq 0$,
		\[x \in \OK \iff x^{-1} \notin I_K,\]
		hence $\OK$ is $\bigvee$-definable over $K$. Consequently, the induced valuation $v_K$ is $K$-invariant: $v_K(x)$ depends only on $\tp(x/K)$.
	\end{enumerate}
\end{remark}

Now suppose $\CCC \succ \CC$ is $\lambda$-saturated and $\lambda$-strongly homogeneous for (strong limit) $\lambda > |\CC|$. Working with $\CC$ in place of $K$ we can repeat the whole construction and consider a partial type $I_\CC$ over $\CC$ along with the valuational ring $\OCC \subseteq \CCC$ and the valuation $v_\CC : \CCC^{\times} \to \CCC^{\times}/\OCC^{\times}$. Then a global type $p \in S(\CC)$ is infinitesimal if $p$ extends $I_\CC$. A crucial result is the following:
\begin{fact}[\cite{Johnson}, Corollary 4.10]
	\label{fact:boundedInf}
	There are only boundedly many global infinitesimal types.
\end{fact}

We now proceed to retrieve the consequences for topological dynamics.

\begin{proposition}
	\label{prop:inffgen}
	Every nonzero global infinitesimal type is a multiplicative $f$-generic.
\end{proposition}
\begin{proof}
	The set $[I_\CC] \cap S(\CC)$ is closed under multiplication by $\CC$ by Fact \ref{fact:infClosed}(ii) and bounded in size by Fact \ref{fact:boundedInf}. Hence the multiplicative orbit of each nonzero global infinitesimal is bounded and consists only of infinitesimal types.
\end{proof}

\begin{proposition}
	The partial type $I_K$ is $0$-definable and its unique global heir is $I_\CC$.
\end{proposition}
\begin{proof}
	Let $\phi(x, \bar{z})$ be a formula without parameters and
	\[\psi(x, \bar{z}) = \exists^{\infty} y ~ \phi(y, \bar{z}) \land \phi(x + y, \bar{z}).\]
	Then $I_K = \{\psi(x, \bar{z}) : \bar{z} \subset K,\ \exists^{\infty} y ~ \phi(y, \bar{z})\}$ and $I_\CC = \{\psi(x, \bar{z}) : \bar{z} \subset \CC,\ \exists^{\infty} y ~ \phi(y, \bar{z})\}$.
	Clearly $\exists^{\infty} y ~ \phi(y, \bar{z})$ is the $\psi(x, \bar{z})$-definition of $I_K$ over $\emptyset$ and the same is true for $I_\CC$.
\end{proof}

We obtain the following corollary about heir extensions of infinitesimal types over a small model:
\begin{corollary}
	Let $p \in S(K)$ be infinitesimal and $p'$ its global heir. Then $p'$ is infinitesimal (over $\CC$).
\end{corollary}

We conclude with the previously announced result.:
\begin{proof}[of Theorem \ref{thm:dp minimal field}]
	Let $p$ be a nonzero global infinitesimal type. Its restriction $r = p \restriction K \in S(K)$ is infinitesimal and definable over $K$. Then the global heir of $r$ is also $K$-definable and infinitesimal over $\CC$, hence a multiplicative f-generic by Proposition \ref{prop:inffgen}.
\end{proof}

We point out that there is a dp-rank $2$ field $K$ definable in a model over which all types are definable such that its multiplicative group is not \emph{dfg}. 
\begin{example}
	Let $\mathbb{C}$ be the field of complex numbers equipped with a predicate $R$ for the reals. Then its multiplicative group $\Gm$ is definably isomorphic to the product $S^1 \times (R^+, \cdot)$ where $S^1$ is the circle group. Then $(S^1 \times (R^+, \cdot))^0 = S^1 \times (R^+, \cdot)$, but $(S^1 \times (R^+, \cdot))^{00} = (S^1)^{00} \times (R^+, \cdot) \lneq S^1 \times (R^+, \cdot)$, since $(S^1)^{00}$ is a proper subgroup of $S^1$ consisting of all elements infinitesimally close to identity over $\mathbb{R}$. Hence $\Gm$ does not have \emph{dfg}: by an observation by Yao and Pillay, if $G$ has \emph{dfg}, then $G^{00} = G^{0}$.
\end{example}

We finish this section by giving examples of definable algebraic groups with \emph{dfg}. Recall that a (connected) unipotent algebraic over $K$ group $G$ has a normal sequence $\{1\} = N_0 \lhd N_1 \lhd \ldots \lhd N_n = G$ such that each quotient $N_{i+1}/N_i$ is (definably) isomorphic to $\Ga$. By Proposition \ref{prop:dfg extension} and induction, we have (without assuming dp-minimality of $K$):
\begin{corollary}
	Let $K$ be a NIP field with all types over $K$ definable. Then every connected unipotent algebraic over $K$ group $U$ has \emph{dfg} and is definably extremely amenable.
\end{corollary}
Recall now that a \emph{triagonizable} group is definably isomorphic to $U \ltimes \mathbb{G}^n_m$ for some unipotent algebraic group $U$. If the characteristic of the field is $0$, then $U$ is connected. Hence:
\begin{corollary}
	Let $K$ be a dp-minimal field of characteristic $0$ with all types over $K$ definable. Then every triagonizable group $G$ over $K$ has \emph{dfg}.
\end{corollary}
This extends Corollary 2.8 of \cite{YaoTri}.

	\bibliographystyle{plain}
	\bibliography{linear}

\begin{thebibliography}{10}

\bibitem{CS}
Artem Chernikov and Pierre Simon.
\newblock Definably amenable {NIP} groups.
\newblock {\em Journal of the American Mathematical Society}, 31(3):609--641,
  2018.

\bibitem{GPP}
Jakub Gismatullin, Davide Penazzi, and Anand Pillay.
\newblock Some model theory of {$SL(2, \mathbb{R})$}.
\newblock {\em Fundamenta Mathematicae}, 229:117--128, 2014.

\bibitem{HP}
Ehud Hrushovski and Anand Pillay.
\newblock On {NIP} and invariant measures.
\newblock {\em Journal of the European Mathematical Society}, 13(4):1005--1061,
  2011.

\bibitem{Jag1}
Grzegorz Jagiella.
\newblock Definable topological dynamics and real {Lie} groups.
\newblock {\em Mathematical Logic Quarterly}, 61(1-2):45--55, 2015.

\bibitem{Jag2}
Grzegorz Jagiella.
\newblock The {Ellis} group conjecture and variants of definable amenability.
\newblock {\em The Journal of Symbolic Logic}, 83(4):1376--1390, 2018.

\bibitem{Johnson}
Will Johnson.
\newblock On dp-minimal fields.
\newblock {\em arXiv preprint arXiv:1507.02745}, 2015.

\bibitem{Kirk}
Thomas Kirk.
\newblock {Definable} {Topological} {Dynamics} of {$SL_2(\mathbb{C}((t))$}.
\newblock {\em arXiv preprint arXiv:1903.03570}, 2019.

\bibitem{Krup}
Krzysztof Krupi{\'n}ski.
\newblock Definable topological dynamics.
\newblock {\em The Journal of Symbolic Logic}, 82(3):1080--1105, 2017.

\bibitem{KPbohr}
Krzysztof Krupi\'nski and Anand Pillay.
\newblock Generalised {Bohr} compactification and model-theoretic connected
  components.
\newblock In {\em Mathematical Proceedings of the Cambridge Philosophical
  Society}, volume 163, page 219. Cambridge University Press, 2017.

\bibitem{New1}
Ludomir Newelski.
\newblock Topological dynamics of definable group actions.
\newblock {\em The Journal of Symbolic Logic}, 74(1):50--72, 2009.

\bibitem{New2}
Ludomir Newelski.
\newblock Model theoretic aspects of the {Ellis} semigroup.
\newblock {\em Israel Journal of Mathematics}, 190(1):477--507, 2012.

\bibitem{New3}
Ludomir Newelski.
\newblock Topological dynamics of stable groups.
\newblock {\em The Journal of Symbolic Logic}, 79(4):1199--1223, 2014.

\bibitem{PPYpadic}
Davide Penazzi, Anand Pillay, and Ningyuan Yao.
\newblock Some model theory and topological dynamics of $ p $-adic algebraic
  groups.
\newblock {\em Fundamenta Mathematicae}, 247:191--216, 2019.

\bibitem{Pil}
Anand Pillay.
\newblock Topological dynamics and definable groups.
\newblock {\em The Journal of Symbolic Logic}, 78(2):657--666, 2013.

\bibitem{PYpadic}
Anand Pillay and Ningyuan Yao.
\newblock Definable $f$-{Generic} {Groups} over $p$-{Adic} {Numbers}.
\newblock {\em arXiv preprint arXiv:1911.01833}, 2019.

\bibitem{YaoTri}
Ningyuan Yao.
\newblock Definable topological dynamics for trigonalizable algebraic groups
  over {$\mathbb{Q}_p$}.
\newblock {\em Mathematical Logic Quarterly}, 65(3):376--386, 2019.

\bibitem{Yao1}
Ningyuan Yao and Dongyang Long.
\newblock Topological dynamics for groups definable in real closed field.
\newblock {\em Annals of Pure and Applied Logic}, 166(3):261--273, 2015.

\end{thebibliography}
\end{document}